\documentclass[12pt,a4paper]{amsart}

\usepackage{latexsym} 
 
\usepackage[dvips]{graphics}
\usepackage{epsfig}
\usepackage{mathrsfs}
\usepackage{amssymb}
\usepackage{amsthm}
\usepackage{amsfonts}
\usepackage{amsmath}
\usepackage{amstext}
\usepackage{amscd}
\usepackage{tikz}
\usepackage{epic}
\usepackage{eepic}
\usepackage{pstricks,pst-plot}
\usepackage{enumerate}

\numberwithin{equation}{section}

\theoremstyle{plain}%default
\newtheorem{thm}{Theorem}[section] 
\newtheorem{prop}[thm]{Proposition}
\newtheorem{cor}[thm]{Corollary}
\newtheorem{lem}[thm]{Lemma}

\newtheorem{theorem*}{Theorem}[]

\theoremstyle{definition}
\newtheorem{defn}[thm]{Definition}

\newtheorem{example}[thm]{Example}

\theoremstyle{remark}
\newtheorem{rem}[thm]{Remark}

\newcommand{\N}{\mathbb{N}}
\newcommand{\R}{\mathbb{R}}

\def\accentclass@{7}
\def\makeacc@#1#2{\def#1{\mathaccent"\accentclass@#2 }}
\makeacc@\cir{017}

\title[biLipschitz homeomorphism]
{Notes on (SSP) sets}

\author{Satoshi Koike and Laurentiu Paunescu}
\dedicatory{}

\address{Department of Mathematics, Hyogo University of Teacher Education,
Kato, Hyogo 673-1494, Japan}
\email{koike@hyogo-u.ac.jp} 
\address{School of Mathematics, University of Sydney, Sydney, NSW, 2006,
Australia}
\email{laurent@maths.usyd.edu.au}

\date{\today}
\begin{document}

\thanks{}

%%%%%%%%%%%%%%%%%%%%%%%%%%%%%%%%%%%%%%%%%%%%%%%%%%%%%%%%%%%%%%%%%%%%%%%%%%%%%%%\maketitle

\begin{abstract}

In \cite{kp4} we investigate the directional behaviour of bi-Lipschitz homeomorphisms $h : (\R^n,0) \to (\R^n,0)$ 
for which there exist  the limits $\lim_{n\to \infty} nh(\frac{x}{n})$, denoted by $\overline{h}(x)$. 
The existence of  such $\overline{h}(x)$ makes trivial to see that $\overline{h}(D(A)) = D(h(A))$ 
for arbitrary set-germs $A$ at $0 \in \R^n$.

Recently, J. Edson Sampaio made the remarkable observation
%the very nice observation
(\cite{sampaio}) that we always can assume the existence of a subsequence $n_i\in \N$, 
such that $\lim_{n_i\to \infty} n_ih(\frac{x}{n_i})=dh(x)$ (in his notation) and this $dh$, although not so strong as $\overline{h}$, 
behaves as well directional-wise for subanalytic sets.
He uses this fact to show that bi-Lipschitz homeomorphic subanalytic sets have bi-Lipschitz homeomorphic tangent cones.

The purpose of this note is to show that Sampaio's $dh$ works as well for  (SSP) sets, that is, the above result is 
characteristic  for (SSP) sets,  a much wider class. 
In particular, we show that the transversality between (SSP) sets is preserved
under bi-Lipschitz homeomorphisms (see \ref{trans}).

\end{abstract}

%%%%%%%%%%%%%%%%%%%%%%%%%%%%%%%%%%%%%%%%%%%%%%%%%%%%%%%%%%%%%%%%%%%%%
\maketitle
\bigskip
\section{Introduction.}\label{introduction}

In \cite{kp1}  we proved that the dimension of the common direction set of two subanalytic subsets is a bi-Lipschitz invariant.
In proving that, we introduced and  essentially used the notion of  sequence selection
property, denoted by $(SSP)$ for short. 
%in order to show 
%that the dimension of the common direction set of two
%subanalytic subsets is preserved by a bi-Lipschitz homeomorphism 
%provided that their images are also subanalytic.
Subsequently we have published three more papers \cite{kp2}, \cite{kp3} and \cite{kp4}, where we proved
essential directional properties of sets satisfying 
$(SSP)$ with respect to bi-Lipschitz homeomorphisms.
For instance we proved the transversality theorem in the singular case
and two types of $(SSP)$ structure preserving theorems (\cite{kp3}), and
we introduced the notion of directional homeomorphism,
proving a unified $(SSP)$ structure preserving theorem with 
directional homeomorphisms (\cite{kp4}).

In this note, using Sampoio's idea, we generalise his main result 
in \cite{sampaio} and the aforementioned main result in \cite{kp1} 
to the case of the $(SSP)$ setting. Although the proofs are in the spirit of \cite{kp4} , basically the same as in \cite{sampaio} (at times even simpler), due to the wide potential applications, we believe that it is still worth mentioning this generalisation.

We describe both the notions and notations necessary for this topic and our
results in the $(SSP)$ setting and some numerical properties of $(SSP)$
in \S 2.
In \S 3 we describe the main results in this note
and give their proofs.

\vspace{3mm}

% Throughout this paper we use the following notations:

% \vspace{2mm}

% Let $\{ a_m \}$, $\{ b_m \}$ be sequences of points 
% of $\R^n$ tending to the origin $0 \in \R^n$.
% If there are a natural number $N \in \N$ and 
% a real number $K > 0$ such that
% $$
% \| a_m \| \le K \| b_m \| , \ \ \forall m \ge N
% $$
% then we write $\| a_m \| \precsim \| b_m \|$ 
% (or $\| b_m \| \succsim \| a_m \| 
% $).
% If $\| a_m \| \precsim \| b_m \|$ 
% and $\| b_m \| \precsim \| a_m \| $,\
% we write $\| a_m \| \thickapprox \| b_m \|$.

\bigskip

%%%%%%%%%%%%%%%%%%%%%%%%%%%%%%%%%%%%%%%%%%%%%%%%%%%%%%%%%%%%%%%%%%%%

%%%%%%%%%%%%%%%%%%%%%%%%%%%%%%%%%%%%%%%%%%%%%%%%%%%%%%%%%%%%%%%%%%%%
\section{Directional Properties of Sets}\label{directional property}

In this section we recall the notions of direction set 
and sequence selection property, and also
several elementary properties concerning $(SSP)$.

%%%%%%%%%%%%%%%%%%%%%%%%%%%%%%%%%%%%%%%%%%%%%%%%%%%%%%%%%%%%%%%%%%%%
\subsection{Direction set}\label{direction set}

Let us recall the notion of direction set.

\begin{defn}\label{directionset}
Let $A$ be a set-germ at $0 \in \R^n$ such that
$0 \in \overline{A}$.
We define the {\em direction set} $D(A)$ of $A$ at $0 \in \R^n$ by
$$
D(A) := \{a \in S^{n-1} \ | \
\exists  \{ x_i \} \subset A \setminus \{ 0 \} ,
\ x_i \to 0 \in \R^n  \ \text{s.t.} \
{\frac{x_i}{ \| x_i \| }} \to a, \ i \to \infty \}.
$$
Here $S^{n-1}$ denotes the unit sphere centred at $0 \in \R^n$.
\end{defn}

For a subset $A \subset S^{n-1}$, we denote by $L(A)$
a half-cone of $A$ with the origin $0 \in \R^n$ as the vertex:
$$
L(A) := \{ t a \in \R^n\ | \ a \in A, \ t \ge 0 \}.
$$
For a set-germ $A$ at $0 \in \R^n$ such that $0 \in \overline{A}$, 
we put $LD(A) := L(D(A))$, and call it the {\em real tangent cone}
of $A$ at $0 \in \R^n$.

% Let $U, \ V \subset \R^n$ such that 
% $0\in \overline{U} \cap \overline{V}$. 
% The following properties hold:
% \begin{enumerate} 
% \item $D(\overline{U})=D(U),$
% \item $D(U\cup V)=D(U)\cup D(V),$
% \item $\overline{\cup_iD(U_i)}\subseteq D(\cup U_i),$
% \item If $U_i$ are half-cones then $\overline{\cup_iD(U_i)}= D(\cup U_i),$
% \item  $D(U\cap V)\subseteq  D(U)\cap D(V).$
% \end{enumerate}

%%%%%%%%%%%%%%%%%%%%%%%%%%%%%%%%%%%%%%%%%%%%%%%%%%%%%%%

\subsection{Sequence selection property}
\label{condition (SSP)}

Let us recall the notion of condition  $(SSP)$.
%In fact here we give a generalised notion of (SSP) 
%relatively to a subset of $\R^n$.

\begin{defn}\label{SSP}
Let $A$ be a set-germ at $0 \in \R^n$
such that $0 \in \overline{A}$.
We say that $A$ satisfies {\em condition} $(SSP)$,
if for any sequence of points $\{ a_m \}$ of $\R^n$
tending to $0 \in \R^n,$ such that 
$\lim_{m \to \infty} \frac{a_m}{\| a_m \| } \in D(A)$,
there is a sequence of points $\{ b_m \} \subset A$ such that,
$$
\| a_m - b_m \| \ll \| a_m \|, \ \| b_m \| , 
$$ 
i.e. $\lim_{m \to \infty}{\frac{\| a_m - b_m \|}{ \| a_m \|} } = 0$.  
\end{defn}

% We have the following characterisation of condition $(SSP)$.
% As mentioned in \cite{kp3}, the proof in the relative case is similar to
% the non-relative case,  for  which we gave a detailed proof 
% in \cite{kp2}.

% \begin{lem}\label{SSPWSSP}(\cite{kp3} Proposition 2.7)
% Let $A, B $ be two set-germs at $0 \in \R^n$
% such that $0 \in \overline{A}\cap \overline{B}$.
% If $A$ satisfies condition $(WSSP)$-relative to $B$, 
% then it satisfies condition $(SSP)$-relative to $B$.
% Namely, the conditions relative $(SSP)$ and relative $(WSSP)$ are equivalent.
% \end{lem}

%We first give a concrete example of a set satisfying condition (SSP).
Below we give several general examples of sets satisfying condition (SSP), to illustrate the richness of this class.
Consult \cite{kp3} for more concrete and general examples.

\begin{example}\label{SSP}

\begin{enumerate}[(1)]
\item 
Let $a_m := {\frac{1}{ m}} \in \R$, $m \in \N$, and set $A := \{ a_m \} \subset \R$.
Then $0 \in \overline{A}$ and 
%Let $e_m$ be the middle point of $a_m$ and $a_{m+1}$.
%Then we have
%$$
%|a_m - e_m | = |e_m - a_m | = {\frac{1}{ 2m(m + 1)}}.
%$$
%Therefore we can see that 
$A$ satisfies condition $(SSP)$.
%\end{example}
%\begin{example}\label{remark2}

Let $A \subseteq \R^n$ be a set-germ at $0 \in \R^n$
such that $0 \in \overline{A}$, then the following hold:

\item The cone $LD(A)$ satisfies condition $(SSP)$,

\item If $A$ is subanalytic or definable in some
o-minimal structure, then it satisfies condition $(SSP)$.
See \cite{hironaka} for the definition of subanalytic,
and see \cite{coste, dries} for the definitions
of definable and o-minimal.

\item If $A$ is a finite union of sets, all of which satisfy 
condition $(SSP)$, then $A$ satisfies condition $(SSP)$.

\item If $A$ is a $C^1$ manifold such that $0 \in A$,
then it satisfies condition $(SSP)$ and $LD(A)=T_0(A)$ 
i.e. the tangent space of $A$ at $0 \in \R^n$
(this is not necessarily true for $C^0$ manifolds or if 
$0\notin A$).

\item Let $\pi :\mathcal M_n \to \R^n$ be the blowing-up at $0\in \R^n$. 
%If the strict transform of $A$ by $\pi, \overline{\pi^{-1}(A\setminus \{0\})}$ 
%Let $B\subset\mathcal M_n$ such that $\overline B \cap E \neq \emptyset,  E=\pi^{-1}(0) $. 
 It is not difficult to produce an example  $B$  which satisfies condition (SSP) and   $\pi(B)=A$ does not necessarily satisfy (SSP). For instance we can take $B=C\cup E,  E=\pi^{-1}(0),   C \cap E=\{a\}$, such that $C$ does not satisfy (SSP) and $LD(C)\subset LD(E )$ at $a$. Then $\pi (B)=\pi(C)$ does not satisfy (SSP), whereas $B$ does satisfy (SSP).

%\begin{example}
\item
Let us denote by $\ell$ the positive $x$-axis, and by $m$ the half line
defined by $y = cx$, $x \ge 0$, for some $c > 0$.
There are many types of zigzag curves  $B$ having infinitely many oscillations around
$0 \in \R^2$ between $\ell$ and $m$.
Some of them do not satisfy condition $(SSP)$ 
e.g. Example 3.4 in \cite{kp1}, 
where the union of $B$ and $\ell$ consists of similar triangles. See the example below \ref{zigzag} for more on condition (SSP) and zigzags.

 Let $\pi :\mathcal M_2 \to \R^2$ be the blowing-up at $0\in \R^2$. 
Using a local coordinate of $\mathcal M_2$, $\pi$ is expressed by
$\pi (X,Y) = (XY,Y)$.
%We put $B$ on the $(X,Y)$-plane so that $\ell$ is the positive $Y$-axis.
Let $B$ be as  above, with or without (SSP).
Then we can see that $A := \pi (B)$ is in the region $|x| \le c|y|^2$, $x \ge 0$, $y \ge 0$.
Therefore $LD(A)$ is the positive $y$-axis.
So regardless whether $B$ has (SSP) or not, one can see that  its image  $A=\pi (B)$ satisfies condition (SSP).
Compare to 
%example \ref{SSP} 
(6).
%\end{example}
\item Let $0\in \overline {A}\cap \overline {B}$  and assume that $LD(A)\cap LD(B)=\{0\}$. Then $A\cup B$ has (SSP) if and only if both $A$ and $B$ have (SSP).

%We put $a_n=n^{-\alpha(n)}$ and we may assume that for  all $n\geq 2$ we have $k-1<\alpha(n)\leq k,$ for some integer $k$ for which we have 
%$n^{1-k}>a_n\geq n^{-k}, n \geq 2$. By construction we have $(1-\frac{1}{n+1})^{-\alpha(n)}>(n+1)^{\alpha(n)-\alpha(n+1)}$ and $(1+\frac{1}{n})^{\alpha(n+1)}>n^{\alpha(n)-\alpha(n+1)}$ . The first inequality shows that whenever  $\alpha(n)\geq \alpha(n+1)$ we have $1\geq (n+1)^{-\alpha(n)+\alpha(n+1)}>(1-\frac{1}{n+1})^{\alpha(n)}$ and whenever $\alpha(n)\geq \alpha(n+1)$  the second implies $(1+\frac{1}{n})^{\alpha(n+1)}>n^{\alpha(n)-\alpha(n+1)}\geq 1$,  in turn imply that $n^{\alpha(n)-\alpha(n+1)}\to1, n\to \infty$ and this is equivalent to $\frac{a_n}{a_{n+1}}\to 1$ i.e. $A$ has (SSP).
\end{enumerate}
\end{example}
%\bigskip

We have the following criterion for condition $(SSP)$.

\begin{prop}
$A$ satisfies condition $(SSP)$ if and only if $\text{dist}(ta,A) = 0(t),$  
for any direction $a \in DA$.
\end{prop}

%%%%%%%%%%%%%%%%%%%%%%%%%%%%%%%%%%%%%%%%%%%%%%%%%%%%%%%%%%%%%%%%%%%%%%%%%%%%%%%
\subsection{Numerical properties of sequences of real numbers on (SSP)}

Let us denote by $\mathcal{A}$ the set of strictly decreasing sequences 
$\{ a_m\}$ of positive real numbers tending to $0 \in \R$, namely 
$\{ a_m\}$ satisfies the following:
$$
a_m > a_{m+1} > 0 \ (m \in \N ), \ \ a_m \to 0 \ \text{as} \ m \to \infty .
$$
For an element of $\mathcal{A}$ we have the following criterion
to satisfy condition $(SSP)$.

\begin{lem}\label{SSPcriterion}
For $\{ a_m \} \in \mathcal{A}$, $\{ a_m\}$ satisfies condition $(SSP)$
if and only if
$$
\frac{a_m}{a_{m+1}} \to 1 \ \text{as} \ m \to \infty .
$$
\end{lem}

\begin{proof}
Since $\{ a_m \} \in \mathcal{A}$, $\{ a_m\}$ satisfies condition $(SSP)$
if and only if
$$
\| a_{m+1} - \frac{a_{m+1} + a_m}{2} \| \ll \| a_{m+1} \| .
$$
This condition is equivalent to
$$
1 - \frac{a_m}{a_{m+1}} \to 0 \ \text{as} \ m \to \infty .
$$
Therefore the statement follows.
\end{proof}

We set
$$
\mathcal{A}_{SSP} := \{ \{ a_m \} \in \mathcal{A} \ | \
\{ a_m \} \ \text{satisfies condition} \ (SSP) \ \text{at} \ 0 \in \R \} ,
$$
and define the summation $``+"$ and multiplication $``\cdot"$ for elements 
of $\mathcal{A}_{SSP}$ as follows:
$$
\{ a_m \} + \{ b_m \} := \{ a_m + b_m \} , \ \
\{ a_m \} \cdot \{ b_m \} := \{ a_m b_m \} .
$$
Then, as a corollary of the above lemma, we have the following
properties on $\mathcal{A}_{SSP}$.

\begin{cor}\label{operation}
$\mathcal{A}_{SSP}$ is closed under the summation $``+"$ and multiplication 
$``\cdot"$.
Therefore $\mathcal{A}_{SSP}$ is a semigroup with respect to the $``+"$ and
$``\cdot"$.
\end{cor}

Let us define the notion of polynomial boundedness for an element 
of $\mathcal{A}$ (cf. M. Coste \cite{coste}, L. van den Dries \cite{dries}).

\begin{defn}\label{polynomiallybounded}
Let $\{ a_m \} \in \mathcal{A}$.
We say that $\{ a_m \}$ is {\em polynomially bounded},
if there exists a positive integer $k \in \N$ such that
$a_m \ge \frac{1}{m^k}$ for any $m \ge 2$.
\end{defn} 

\begin{example}\label{logarithm}
Let $a_m : = \log \frac{m + 1}{m}$ for $m \in \N$.
Then we have $\{ a_m \} \in \mathcal{A}$.
Using Lemma \ref{SSPcriterion}, we can easily see that
$\{ a_m \}$ satisfies condition $(SSP)$.
On the other hand, we  have
$$
a_m = 1 + \frac{1}{m} \ge \frac{1}{m^2} \ \ \text{for} \ \ m \ge 2 .
$$
Therefore $\{ a_m \}$ is also polynomially bounded.
\end{example}

We next discuss the relationship between condition $(SSP)$
and polynomial boundedness.

\begin{example}\label{SSPnotPB}
Let $a_m := \frac{1}{m^{\sqrt{m}}}$ for $m \in \N$.
Then we have $\{ a_m \} \in \mathcal{A}$.
Let us compute $\frac{a_m}{a_{m+1}}$:

\begin{eqnarray*}
\frac{a_m}{a_{m+1}} &=& \frac{(m + 1)^{\sqrt{m+1}}}{m^{\sqrt{m}}} \\
&=& (\frac{m + 1}{m})^{\sqrt{m}} (m + 1)^{\sqrt{m+1} - \sqrt{m}} \\
&=& [(1 + \frac{1}{m})^m]^{\frac{1}{\sqrt{m}}} 
(m + 1)^{\frac{1}{\sqrt{m+1} + \sqrt{m}}}.
\end{eqnarray*}
Then we can easily see that
$$
[(1 + \frac{1}{m})^m]^{\frac{1}{\sqrt{m}}} , \ 
(m + 1)^{\frac{1}{\sqrt{m+1} + \sqrt{m}}} \to 1 \ \
\text{as} \ m \to \infty .
$$
Therefore it follows from Lemma \ref{SSPcriterion} that
$\{ a_m \}$ satisfies condition $(SSP)$.
On the other hand, $\{ a_m \}$ is not polynomially bounded.
\end{example}

The example above shows that condition $(SSP)$ does not always imply
polynomially bounded.
On  the other hand, we have the following example
concerning the opposite implication.

\begin{example}\label{PBnotSSP}
Let us choose  a sequence $m_i$, $i\in \N$, such that 
$\frac{ m_i}{m_{i+1} }\to \infty,$ (for instance $m_i=i^i$) and define 
the following sequence:
$a_{m_i}=\frac{1}{m_i^2}, \ a_{m_{i+1}-1}=\frac{1}{(m_i^2+m_i)}$ 
and in between them complete with any decreasing sequence.
Then we  have  
$$
\frac{1}{(m_i+j)}>\frac{1}{m_i^2} \geq a_{m_i+j}\geq \frac{1}{(m_i^2+m_i)}
>\frac{1}{(m_i+j)^2}, \ j=1,..., m_{i+1}-m_i-1.
$$
Therefore we get a decreasing sequence $\{ a_m \} \in \mathcal{A}$ such that
$\frac{1}{m^2} \leq a_m <\frac{1}{m}.$
It follows that $\{ a_m \}$ is polynomially bounded. 
But the ratio $\frac{a_m}{a_{m+1}}$ does not tend to 1 (for instance 
$\frac{a_{m_{i+1}-1}}{a_{m_{i+1}}}\to 0$), 
therefore $\{ a_m \}$ does not satisfy condition $(SSP)$!
\end{example}

The last example raises a very interesting question about the density of 
decreasing sequences so perhaps one has to modify  the definition of 
$(SSP)$ a little to include the polynomially bounded sequences as a subset.

\bigskip

%%%%%%%%%%%%%%%%%%%%%%%%%%%%%%%%%%%%%%%%%%%%%%%%%%%%%%%%%%%%%%%%%%%%

%%%%%%%%%%%%%%%%%%%%%%%%%%%%%%%%%%%%%%%%%%%%%%%%%%%%%%%%%%%%%%%%%%%%
\section{Main results}\label{mainresults}

In this section we describe our main results.
We first make some remarks on Lipschitz extension methods. 

\begin{rem}\label{WhitneyBanach}
We we can always construct 
a global Lipschitz extension of a given Lipschitz mapping  
$f:A \to \R^n, A \subset (X,d)$ to $\tilde f :X\to \R^n$.
Indeed, for a Lipschitz function with constant $L, \
f : A \to \R, \ A \subset X$, $A$ endowed with the induced 
metric from $(X,d)$, we have an extension formula 
(see H. Whitney \cite{Wh} or S. Banach \cite{banach}):

$$
\alpha(x):=\inf_{a\in A}(f(a)+Ld(x,a)).
$$
Similarly one can extend it by 
$$
\beta(x):=\sup_{a\in A}(f(a)-Ld(x,a)).
$$ 
% Note that $\beta(x)\leq \alpha(x)$.
% Any convex combination $t\alpha(x)+(1-t)\beta(x)$, $0 \leq t \leq 1$, 
% also gives a Lipschitz extension.
This construction can be used to extend Lipschitz maps as well, however, 
without preserving the Lipschitz constant.
\end{rem}

%%%%%%%%%%%%%%%%%%%%%%%%%%%%%%%%%%%%%%%%%%%%%%%%%%%%%%%%%%%%%%%%%%%%%%%
\begin{rem}\label{doubling}
For a given Lipschitz mapping $f:\R^n\to\R^n$ we can associate the  bi-Lipschitz mappings 
$Y_+(f):\R^n\times\R^n\to \R^n\times\R^n$ defined by $Y_+(f)(x,y):=(x,y+f(x))$ and 
$Y_-(f):\R^n\times\R^n\to \R^n\times\R^n$ defined by $Y_-(f)(x,y):=(x+f(y),y)$.

Given a bi-Lipschitz mapping $\phi:A\to B, A,B\subset \R^n$, we can extend both 
$\phi$ and $\phi^{-1}$ to $\R^n$, say to global Lipschitz mappings $\tilde \phi$ and  $\tilde \phi^{-1}$, 
and then consider the corresponding bi-Lipschitz mappings $Y_+, Y_-$. 

We then consider
the globally defined bi-Lipschitz
$\tilde {\tilde \phi}:=Y_-(\tilde \phi^{-1})^{-1}\circ Y_+(\tilde \phi)$  and note that 
$\tilde {\tilde \phi}(x,0)=(0,\phi(x)), \forall x\in A$.  

In other words, in considering 
the direction cones  of $A$ and $B$  we may assume that $\phi:A\to B$, $A,B\subset \R^n$, is defined globally, see \cite{sampaio}. 
This is a standard way of creating bi-Lipschitz mappings, 
we call it the {\em doubling process}.
\end{rem}

%\bigskip

\begin{rem}\label{der}
Given a bi-Lipschitz homeomorphism $\psi:\R^n \to \R^n$, 
one can consider $\psi_n(x)=n\psi(\frac{x}{n})$ and observe that in a compact neighbourhood of the origin 
we can, via Arzela-Ascoli's theorem, claim the existence of a limit $\psi_{n_i}\to d\psi$ (see \cite{sampaio}). 
This will be bi-Lipschitz as well with the same constants. 
\end{rem}
%%%%%%%%%%%%%%%%%%%%%%%%%%%%%%%%%%%%%%%%%%%%%%%%%%%%%%%%%%%%%%%%%%%%%%%%%%%%%%
We have the following lemma in the $(SSP)$ setting.

\begin{lem}\label{mainlemma}
Let $A$, $B \subset \R^n$ set-germs at $0 \in \R^n$ such that $0 \in \overline{A} \cap \overline{B}$,
and let $\phi : A \to B$ be a bi-Lipschitz homeomorphism. 
If $A$ satisfies condition $(SSP)$, then $d\tilde {\tilde \phi}(LD(A)) \subset LD(B)$.
\end{lem}

\begin{proof} 
By remark \ref{doubling} we may assume that $\phi$ is global and by remark \ref{der} 
we can consider the associated $d\phi=\lim_{i\to\infty}\phi_{n_i}$.
Take an arbitrary $v\in LD(A)$.
Since $A$ satisfies condition $(SSP)$,  there is a sequence of points  $v_i \in A$, $i \in \N$, such that 
$$
\|v_i-\frac{v}{n_i}\| << \frac{1}{n_i}\thickapprox \|v_i\| .
$$ 
Accordingly we have 
$$
\|\phi(v_i)-\phi(\frac{v}{n_i})\|<<\frac{1}{n_i},
$$
which in turn shows that 
$$
\|n_i\phi(v_i)-n_i\phi(\frac{v}{n_i})\|\to 0 \ \ \text{as} \ \  i\to \infty.
$$ 
It follows that $d\phi(v)\in LD(B)$ as claimed (see \cite{sampaio}).
\end{proof}

\begin{rem}
In fact one can prove the following.  Let $A$, $B \subset \R^n$ be set-germs at $0 \in \R^n$ 
such that $0 \in \overline{A} \cap \overline{B}$, and let $\phi : (\R^n,A ,0)\to (\R^n,B,0)$ 
be a Lipschitz mapping-germ.
If $A$ satisfies condition $(SSP)$, then $d\phi (LD(A)) \subset LD(B)$. 
Here $d\phi$ is merely Lipschitz.
\end{rem}

As a corollary of the above lemma, we have the generalised result 
of Theorem 3.2 in \cite{sampaio} to the case in the $(SSP)$ setting.

\begin{thm}\label{Tcone}
Let $A$, $B \subset \R^n$ be set-germs at $0 \in \R^n$ such that $0 \in \overline{A} \cap \overline{B}$,
and let $\phi : A \to B$ be a bi-Lipschitz homeomorphism. 
If both $A, B$ satisfy condition $(SSP)$, then $d\tilde {\tilde \phi}(LD(A)) = LD(B)$.
\end{thm}

\begin{example}\label{zigzag}
Let $f : (\R ,0) \to (\R, 0)$ be a continuous zigzag function like in  Figure $1$,
whose graph has infinitely many oscillations around $0 \in \R^2$
between the positive $x$-axis $\ell$ and the half line $m$ 
defined by $y = cx$, $x \ge 0$, for some $c > 0$.

\begin{figure}[htb]
\centering
\includegraphics[width=.3\linewidth]{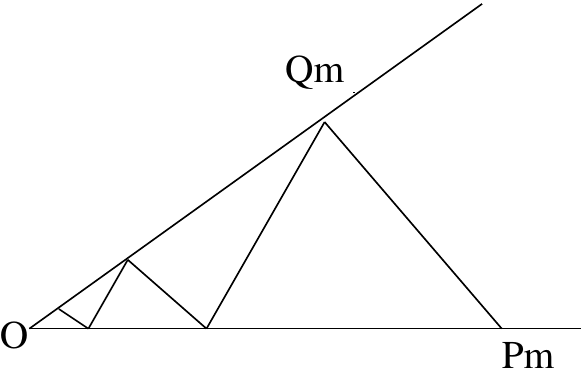}
%\caption{$x^2-y^3$}
\hfil
%\includegraphics[width=.3\linewidth]{cusp02}
%\caption{$x^4-y^5$}
\caption{}
\label{fig:zigzag}
\end{figure}

Now we define the mapping $\phi : (\R^2 ,0) \to (\R^2 ,0)$ by
$$
\phi (x,y) =Y_+(f)= (x,y + f(x)).
$$
Then $\phi$ is a homeomorphism.
Let us remark that $\ell$ satisfies condition $(SSP)$,
$LD(\ell ) = \ell$ and $LD(\phi (\ell ))$ is the sector
surrounded by $\ell$ and $m$ with $0 \in \R^2$ as the vertex.
Therefore, by Theorem \ref{Tcone}, we can see the following property:

\vspace{3mm}

{\em If the zigzag $\phi (\ell )$ satisfies condition $(SSP)$,
then $\phi$ cannot be bi-Lipschitz (i.e. $f$ cannot be Lipschitz).
In other words, if $\phi$ is a bi-Lipschitz homeomorphism ($f$ is Lipschitz),
then the zigzag $\phi (\ell )$ does not satisfy condition $(SSP)$.}

\vspace{3mm}
 
The above property follows also from some directional property
of intersection set (Proposition 2.29 and Appendix in \cite{kp3})
or an important property concerning $LD(h(A)) = LD(h(LD(A)))$ in \cite{kp1}.
\end{example}

%\bigskip
Using Theorem \ref{Tcone} , we can show the following corollaries.

\begin{cor}\label{corollary}
Let $A$ be a set germ at $0 \in \R^n$ such that $0\in \overline{A\setminus \{ 0 \} }$,
and let $0 \in \R^n$ have a neighbourhood in $A$ bi-Lipschitz homeomorphic to an open set 
in some Euclidean space $\R^k$.
Then $LD(A)$ is bi-Lipschitz homeomorphic to $\R^k$.
\end{cor}

\begin{proof} Assume that $A$ is bi-Lipschitz homeomorphic to an open set $U\subset \R^k$. 
Then according to example \ref{SSP} (5), $U$ satisfies condition (SSP) and $LD(U)=\R^k$. 
Therefore by Theorem  \ref{Tcone}, their tangent cones are bi-Lipschitz homeomorphic as well.
\end{proof}

\begin{cor}\label{corollary1}
Let $A$ be a set germ at $0 \in \R^n$ such that $0\in \overline{A\setminus \{ 0 \} }$,
and let $0 \in \R^n$ have a neighbourhood $V$ in $A$ bi-Lipschitz homeomorphic to a cone $LD(C)$.
Then $V$ and $LD(A)$ are bi-Lipschitz homeomorphic as well, in particular $\text{dim} D(A)=\text{dim} A-1$.
\end{cor}

\begin{proof}  Any cone has (SSP) and $LD(LD(C))=LD(C)$.
Therefore by Theorem  \ref{Tcone}, their tangent cones are bi-Lipschitz homeomorphic as well.
\end{proof}

\bigskip

We can show also the following lemma in the $(SSP)$ setting.

\begin{lem}
Let $A$, $B \subset \R^n$ be set-germs at $0 \in \R^n$ such that $0 \in \overline{A} \cap \overline{B}$,
and let $h : \R^n\to\R^n$ be a bi-Lipschitz homeomorphism. 
If both $A, B$ satisfy condition $(SSP)$, then
$$
\dim (D(h(A)) \cap D(h(B))) \geq \dim (D(A) \cap D(B)).
$$
\end{lem}

\begin{proof}
Having established Lemma \ref{mainlemma}, the proof follows as in \cite{sampaio}.
\end{proof}

As a consequence of the above lemma, we have the generalised result
of Main Theorem in \cite{kp1} to the case in the $(SSP)$ setting.

\begin{thm}\label{dim}
Let $A$, $B \subset \R^n$ be set-germs at $0 \in \R^n$ such that 
$0 \in \overline{A} \cap \overline{B}$, and let 
$h : (\R^n,0) \to (\R^n,0)$ be a bi-Lipschitz homeomorphism.
Suppose that $A, \ B, \ h(A), \ h(B)$ satisfy condition $(SSP)$.
Then we have the equality of dimensions,
$$
\dim (D(h(A)) \cap D(h(B))) = \dim (D(A) \cap D(B)).
$$
\end{thm}
%\begin{proof}
%\end{proof}
%%%%%%%%%%%%%%%%%%%%%%%%%%%%%%%%%%%%%%%%%%%%%%%%%%%%%%%%%%%%%%%%%%%%
\begin{defn}
Let $A$, $B \subset \R^n$ be set-germs at $0 \in \R^n$ such that $0 \in \overline{A} \cap \overline{B}$. We say that $A, B$ are {\it transverse} at $0\in \R^n$ if and only if:

$$ \text{dim} LD(A) +\text{dim LD(B)}- \text{dim}(LD(A)\cap LD(B)) =n.$$

\end{defn}

As a corollary of Theorem \ref{dim}, we have the following preserving of transversality result.
\begin{cor}\label{trans}
Let $A$, $B \subset \R^n$ be set-germs at $0 \in \R^n$ such that 
$0 \in \overline{A} \cap \overline{B}$, and let 
$h : (\R^n,0) \to (\R^n,0)$ be a bi-Lipschitz homeomorphism.
Suppose that $A, \ B, \ h(A), \ h(B)$ satisfy condition $(SSP)$. Then $A$ and $B$ are {\it transverse} at $0\in \R^n$ if and only if $h(A)$ and $h(B)$ are {\it transverse} at $h(0)=0\in \R^n.$

\end{cor}

On the other hand in \cite{kp3} we introduced a notion of weak transversality and showed in Theorem 3.5 that weak transversality is preserved under rather mild assumptions. We are going to recall the result for reader convenience.

\begin{defn}
Let $A$, $B \subset \R^n$ be set-germs at $0 \in \R^n$ such that $0 \in \overline{A} \cap \overline{B}$. We say that $A, B$ are {\it weakly transverse} at $0\in \R^n$ if and only if
$ D(A) \cap D(B)=\emptyset$.
\end{defn}

\begin{thm}
Let $A$, $B$ be two set-germs at $0 \in \R^n$
such that $0 \in \overline{A} \cap \overline{B}$,
and let $h : (\R^n,0) \to (\R^n,0)$ be a bi-Lipschitz homeomorphism.
Suppose that $A$ or $B$ satisfies condition $(SSP)$,
and $h(A)$ or $h(B)$ satisfies condition $(SSP)$.
Then $A$ and $B$ are weakly transverse at $0 \in \R^n$ if and only 
if $h(A)$ and $h(B)$ are weakly transverse at $0 \in \R^n$.
\end{thm}
%\medskip
%%%%%%%%%%%%%%%%%%%%%%%%%%%%%%%%%%%%%%%%%%%%%%%%%%%%%%%%%%%%%%%%%%%%%%%%%

\end{document}